\documentclass[11pt]{article}

\usepackage{latexsym,amssymb}

\begin{document}
\title{\bf The Minimum Number of Dependent \\ 
Arcs and a Related Parameter of \\ 
Generalized Mycielski Graphs}

\author{Hsin-Hao Lai\\
\normalsize Department of Mathematics\\
\normalsize National Kaohsiung Normal University \\
\normalsize Yanchao, Kaohsiung 824, Taiwan\\
\normalsize {\tt Email:hsinhaolai@nknucc.nknu.edu.tw}\\
\and
Ko-Wei Lih\\
\normalsize Institute of Mathematics\\
\normalsize Academia Sinica\\
\normalsize Nankang, Taipei 115, Taiwan\\
\normalsize {\tt Email:makwlih@sinica.edu.tw}}

\date{\small January 21, 2010}
\maketitle

\newcommand{\cqfd}{\hfill \rule{8pt}{9pt}}
\newcommand{\la}{\langle}
\newcommand{\ra}{\rangle}
\newtheorem{define}{Definition}
\newtheorem{proposition}[define]{Proposition}
\newtheorem{theorem}[define]{Theorem}
\newtheorem{lemma}[define]{Lemma}
\newtheorem{remark}[define]{Remark}
\newtheorem{corollary}[define]{Corollary}
\newtheorem{problem}[define]{Problem}
\newtheorem{conjecture}[define]{Conjecture}
%
%
\newenvironment{proof}{
\par
\noindent {\bf Proof.}\rm}%
{\mbox{}\hfill\rule{0.5em}{0.809em}\par}

\baselineskip=16pt
\parindent=0.5cm

\begin{abstract}
\noindent
Let $D$ be an acyclic orientation of the graph $G$. An arc of 
$D$ is {\em dependent} if its reversal creates a directed cycle. 
Let $d_{\min}(G)$ denote the minimum number of dependent arcs 
over all acyclic orientations of $G$. Let $G(V_0, E_0)$ be a 
graph with vertex set $V_0 = \{\la 0,0 \ra, \la 0,1 \ra, \ldots , 
\la 0,n-1 \ra\}$ and edge set $E_0$. The {\em generalized Mycielski 
graph} ${\sf M}_m(G)$ of $G$, $m > 0$, has vertex set $V=V_0 \cup 
(\cup_{i=1}^{m} V_i) \cup \{u\}$, where $V_i=\{\la i,j \ra \mid 0 
\leqslant j \leqslant n-1\}$ for $1 \leqslant i \leqslant m$, and 
edge set $E=E_0 \cup (\cup_{i=1}^{m} E_i) \cup \{\la m,j \ra u \mid 
0 \leqslant j \leqslant n-1\}$, where $E_i= \{\la i-1,j \ra \la i,k 
\ra \mid \la 0,j \ra \la 0,k \ra \in E_0\}$ for $1 \leqslant i 
\leqslant m$. We generalize results concerning $d_{\min}({\sf M}_1(G))$ 
in K. L. Collins, K. Tysdal, {\em J. Graph Theory}  46 (2004),  285-296,
to $d_{\min}({\sf M}_m(G))$. The underlying graph of a Hasse diagram 
is called a {\em cover graph}. Let $c(G)$ denote the the minimum number 
of edges to be deleted from a graph $G$ to get a cover graph. Analogue 
results about $c(G)$ are also obtained.

\noindent
{\em Keyword}.\  acyclic orientation, dependent arc, source-reversal, 
cover graph, generalized Mycielski graph

\medskip
\noindent
{\em MSC 2000}:\  05C99
\end{abstract}

%
\section{Introduction}
%

Graphs considered in this paper are finite, without loops, or
multiple edges. We use $|G|$ and $\|G\|$, respectively, to denote 
the cardinalities of vertex set $V$ and edge set $E$ of a graph 
$G(V,E)$. The degree of a vertex $v$ of $G$ is denoted $d_G(v)$.
An orientation $D$ of $G$ is obtained by assigning an arbitrary 
direction, either $x \rightarrow y$ or $y \rightarrow x$, on every 
edge $xy$ of $G$. The original undirected graph is called the 
{\em underlying graph} of any such orientation. {\em Sources} 
(or {\em sinks}) are vertices with no ingoing (or outgoing) arcs.
An orientation $D$ is called {\em acyclic} if there does not exist 
any directed cycle.

Suppose that $D$ is an acyclic orientation of $G$. An arc $u 
\rightarrow v$ of $D$, or its underlying edge, is called {\em 
dependent} (in $D$) if the new orientation $D'=(D-(u \rightarrow 
v)) \cup (v \rightarrow u)$ contains a directed cycle. Note that 
$u \rightarrow v$ is a dependent arc if and only if there exists 
a directed walk of length at least 2 from $u$ to $v$. Let $d(D)$ 
denote the number of dependent arcs in $D$. Let $d_{\min}(G)$ and 
$d_{\max}(G)$, respectively, denote the minimum and maximum values 
of $d(D)$ over all acyclic orientations $D$ of $G$. It is known 
(\cite{fflw}) that $d_{\max}(G)=\|G\|-|G|+ k$ for a graph $G$ having 
$k$ components.

Let $\chi(G)$ denote the {\em chromatic number} of $G$, i.e., the 
least number of colors to color the vertices of $G$ so that adjacent 
vertices receive distinct colors. Let $g(G)$ denote the {\em girth} 
of $G$, i.e., the length of a shortest cycle of $G$ if there is any, 
and $\infty$ if $G$ possesses no cycles. Fisher et al. \cite{fflw} 
showed that $d_{\min}(G)=0$ when $\chi(G) < g(G)$. The Hasse diagram 
of a finite partially ordered set depicts the covering relation of 
elements; its underlying graph is called a {\em cover graph}. Pretzel 
\cite{pret} proved that $d_{\min}(G)=0$ is equivalent to $G$ being a 
cover graph. 

A convenient tool for us is the source-reversal operation first
introduced by Mosesian in the context of finite posets and extensively 
used by Pretzel dealing with cover graphs. We will summarize the main 
properties of this operation in Section 2. In Section 3, we will 
introduce another parameter $c(G)$ which lower bounds $d_{\min}(G)$ 
and show that $c(G)=1$ if and only if $d_{\min}(G)=1$. In Section 4, 
we will characterize the case $d_{\min}({\sf M}_m(G)) \geqslant 1$. 
We give generalizations of results established by Collins and Tysdal 
in Section 5. In the final Section, we derive upper bounds for 
$c({\sf M}_m(G))$.

%
\section{Source-reversal}
%

Let $u$ be a source of the acyclic orientation $D$. A {\em source-reversal} 
operation applied to $u$ reverses the direction of all outgoing arcs from 
$u$ so that $u$ becomes a sink. The new orientation remains acyclic. Note 
that, if there are no dependent arcs in $D$, neither will there be any after 
a source-reversal.

\begin{theorem}\label{mosesian}
Let $D$ be an acyclic orientation of a connected graph $G$. For any vertex 
$u$ of $G$, there exists an orientation $D'$ of $G$ obtained from $D$ by a 
sequence of source-reversals so that $u$ becomes the unique source of $D'$. 
\end{theorem}

The above result originally appeared in Mosesian \cite{mo}. It was put to 
good use by Pretzel in a series of papers (for example, \cite{pret}, 
\cite{pret86}, \cite{pret03}, and \cite{py}).

Let $D$ be an acyclic orientation of the graph $G$. For an undirected 
cycle $C$ of $G$, we choose one of the two traversals of $C$ as the 
positive direction. An arc is said to be {\em forward} if its orientation 
under $D$ is along the positive direction of $C$, otherwise it is said to 
be {\em backward}. We use $(C,D)^+$ (or $(C,D)^-$) to denote the set of 
all forward (or backward) arcs of $C$ with respect to $D$. The {\em flow 
difference} of $C$ with respect to $D$, denoted $f_D(C)$, is defined to 
be $|(C,D)^+|-|(C,D)^-|$. The cycle $C$ is called {\em $k$-good} if 
$|f_D(C)|\leqslant |C|-2k$, i.e., $C$ has at least $k$ forward arcs and 
$k$ backward arcs. An orientation $D$ is called {\em $k$-good} if all 
undirected cycles of its underlying graph $G$ are $k$-good. The set of 
acyclic orientations coincides with the set of 1-good orientations. A 
graph $G$ has a 2-good orientation if and only if $d_{\min}(G)=0$.

The {\em flow difference} of an orientation $D$ is the mapping $f$ from
all cycles of $G$ to integers such that $f(C)=f_D(C)$ for every cycle $C$.
Let $D$ and $D'$ be two orientations of the graph $G$. We say that $D$ is 
an {\em inversion} of $D'$, and vice versa, if $D$ and $D'$ possess the 
same flow difference.

The following appeared in Pretzel \cite{pret86}.

\begin{theorem}\label{equivalent flow difference}
If $D$ and $D'$ are two acyclic orientations of the graph $G$, then the 
following statements are equivalent.
\begin{enumerate}
\item   $D'$ is an inversion of $D$.
\item   $D'$ can be obtained from $D$ by
a sequence of source-reversals.
\end{enumerate}
\end{theorem}

%
\section{The case for $d_{\min} = 1$}
%

We denote by $c(G)$ the the minimum number of edges to be deleted from 
$G$ so that the remaining graph is a cover graph, i.e., 
$$
c(G)= \min\{|F| \mid F \subseteq E(G) \mbox{ and $G-F$ is a cover graph}\}.
$$
Bollob\'{a}s et al. \cite{bbn} first introduced and studied this parameter.
Their results were extended in R\"{o}dl and Thoma \cite{rt}. It was also 
one of the four parameters that give lower bounds to $d_{\min}(G)$ 
investigated in Lai and Lih \cite{ll}. It is straightforward to observe the 
following.

\medskip
\noindent
{\em Fact 1.}\ \  $c(G) \leqslant d_{\min}(G)$. 

\medskip
\noindent
{\em Fact 2.}\ \  A sufficient and necessary condition for $c(G)=0$ is 
$d_{\min}(G)$ $=0$.

\begin{theorem}\label{c=min=1}
A sufficient and necessary condition for $c(G)=1$ is $d_{\min}(G)=1$.
\end{theorem}

\begin{proof}
It follows from Facts 1 and 2 that $d_{\min}(G)=1$ implies $c(G)=1$.

Now let us assume that $c(G)=1$. Then there exists an edge $e=xy$ such 
that $G' = G - e$ has a 2-good orientation $D'$. We may assume that 
there is no directed path from $y$ to $x$ and extend $D'$ to an acyclic 
orientation $D$ of $G$ by adding the arc $x \rightarrow y$.

Since $G$ has no 2-good orientations, $D$ must have at least one 
dependent arc. If $D$ has only one dependent arc, then we are done.
If $D$ has at least two dependent arcs, then each of them must belong 
to a cycle containing $e$.

We claim that $x \rightarrow y$ can not be dependent in $D$. Suppose on 
the contrary that there exists a directed path $x,v_1,v_2,\ldots ,v_s, 
y,$ $s \geqslant 1$, from $x$ to $y$ in $D$. Since $D$ has at least two 
dependent arcs, there is a dependent arc $e'$ in $D$ distinct from $x 
\rightarrow y$, and there exists a cycle $y, u_1, u_2, \ldots ,u_t, x, y$ 
in $G$ such that $e'$ is the only backward arc in this cycle. Consider 
the closed walk $W=x,v_1,v_2,\ldots$, $v_s, y, u_1, u_2$, $\ldots ,u_t, x$.
Reversing $e'$ converts $W$ into a closed directed walk. Hence, $e'$ is a 
dependent arc in $D'$ which contradicts the 2-goodness of $D'$. Therefore, 
$x \rightarrow y$ is not dependent in $D$.

By Theorems \ref{mosesian} and \ref{equivalent flow difference}, we can 
find an inversion $D^\ast$ of $D$ such that $D^\ast$ and $D$ have the 
same flow difference and $y$ is a source in $D^\ast$.

Let $e^\ast$ be an arbitrary dependent arc in $D^\ast$ and $C^\ast$ be a 
cycle of $G$ such that $(C^\ast,D^\ast)^-=\{e^\ast\}$. Then $C^\ast$ must 
pass through the arc $y \rightarrow x$. Otherwise, $|(C^\ast,D^\ast)^-|=
|(C^\ast,D)^-|=|(C^\ast,D')^-|=1$ implies that $e^\ast$ is a dependent arc 
in $D'$, contradicting the 2-goodness of $D'$.

Suppose that $e^\ast$ is different from the arc $y \rightarrow x$. Hence, 
$y \rightarrow x$ belongs to $(C^\ast,D^\ast)^+$. Then the arc $x 
\rightarrow y$ belongs to $(C^\ast,D)^-$. Since $x \rightarrow y$ is not 
dependent in $D$, we have $|(C^\ast,D)^-| \geqslant 2$. By Theorem 
\ref{equivalent flow difference}, $2=2|(C^\ast,D^\ast)^-|=|C^\ast|-
|(C^\ast,D)^+|+|(C^\ast,D)^-|>2$, a contradiction. We conclude that 
$e^\ast$ must be the arc $y \rightarrow x$ in $D^\ast$. Therefore, 
$d_{\min}(G)=d(D^\ast)=1$.
\end{proof}

\bigskip

An immediate consequence of the above Theorem is the following.

\begin{corollary}
If $d_{\min}(G)=2$, then $c(G)=2$.
\end{corollary}

%
\section{Non-cover Mycielski graphs}
%

Let $G(V_0, E_0)$ be a graph with vertex set $V_0 = \{\la 0,0 \ra,$ 
$\la 0,1 \ra,$ $\ldots ,$ $\la 0,n-1 \ra\}$ and edge set $E_0$. For 
$m > 0$, the {\em generalized Mycielski} graph ${\sf M}_m(G)$ of $G$ 
has vertex set $V=V_0 \cup (\cup_{i=1}^{m} V_i) \cup \{u\}$, where 
$V_i=\{\la i,j \ra \mid 0 \leqslant j \leqslant n-1\}$ for $1 
\leqslant i \leqslant m$, and edge set $E=E_0 \cup (\cup_{i=1}^{m} 
E_i) \cup \{\la m,j \ra u \mid 0 \leqslant j \leqslant n-1\}$, where 
$E_i= \{\la i-1,j \ra \la i,k \ra \mid \la 0,j \ra \la 0,k \ra \in 
E_0\}$ for $1 \leqslant i \leqslant m$. We note that ${\sf M}_1(G)$ 
is commonly known as the {\em Mycielskian} $M(G)$ of $G$. It is easy 
to see that if $H$ is a subgraph of $G$, then ${\sf M}_m(H)$ is a 
subgraph of ${\sf M}_m(G)$. The following was proved in Lih et al. 
\cite{tower}.

\begin{theorem}\label{iffeven}
Let $n \geqslant 3$. Then ${\sf M}_m(C_n)$ is a cover graph if and 
only if $n$ is even.
\end{theorem}

This can be generalized as follows.

\begin{theorem}\label{lai}
$d_{\min}({\sf M}_m(G)) \geqslant 1$ if and only if $G$ is not bipartite.
\end{theorem}

\begin{proof}
If $G$ has no edge, then obviously ${\sf M}_m(G)$ is a cover graph. Let 
$G$ be a bipartite graph with at least one edge. Then $\chi({\sf M}_m(G)) 
= 3 < g({\sf M}_m(G))$. Hence ${\sf M}_m(G)$ is a cover graph. If $G$ is 
not bipartite, then $G$ contains an odd cycle $C$ of length at least 3. 
By Theorem \ref{iffeven}, ${\sf M}_m(C)$ is not a cover graph. Since 
${\sf M}_m(G)$ is a supergraph of ${\sf M}_m(C)$, it is not a cover graph.
\end{proof}

\begin{corollary}
$c({\sf M}_m(G)) \geqslant 1$ if and only if $G$ is not bipartite.
\end{corollary}

We are going to construct examples to show that equality can hold in 
Theorem \ref{lai}.

\begin{theorem}\label{mycielski d_min=1}
Let $G(V_0, E_0)$ be a triangle-free graph that is not bipartite.
Suppose that there exists some vertex $\la 0,v\ra$ of $G$ such that 
$G-\la 0,v\ra$ is a bipartite graph whose two parts are denoted by 
$X$ and $Y$. If $\la 0,v\ra$ has precisely one neighbor in $X$ and 
at least one neighbor in $Y$, then $d_{\min}({\sf M}_m(G))=1$.
\end{theorem}

\begin{proof}
By Theorem \ref{lai}, we know $d_{\min}({\sf M}_m(G))\geqslant 1$.
It suffices to construct an acyclic orientation of ${\sf M}_m(G)$
possessing a unique dependent arc.

{\bf Step 1}.\
Define an orientation $D_1$ of $G$ as follows.

(1)\  If $xy$ is an edge in $G-\la 0,v\ra$, $x \in X$ and $y \in Y$,
then let $x\rightarrow y$.

(2)\  If $\la 0,v'\ra$ is the unique neighbor of $\la 0,v\ra$ in $X$,
then let $\la 0,v'\ra\rightarrow \la 0,v\ra$.

(3)\  If $\la 0,v''\ra$ is any neighbor of $\la 0,v\ra$ in $Y$,
then let $\la 0,v\ra\rightarrow \la 0,v''\ra$.

Obviously, each vertex in $X$ is a source, each vertex in $Y$ is a 
sink, and $\la 0,v\ra$ is neither a source nor a sink. It follows 
that $D_1$ is an acyclic orientation. Moreover, if $P$ is a directed 
path of length at least 2 in $D_1$, then $\la 0,v'\ra$ must be the 
initial vertex of $P$ and the length of $P$ is precisely 2. Since $G$ 
is triangle-free, $D_1$ has no dependent arc.

{\bf Step 2}.\
Let $D_2$ be the extension of $D_1$ into ${\sf M}_m(G)-u$ by defining
$\la i,w_1 \ra \rightarrow \la i-1,w_2 \ra$ and $\la i-1,w_1 \ra 
\rightarrow \la i,w_2 \ra$ if $\la 0,w_1\ra \rightarrow \la 0,w_2\ra$ 
in $D_1$ and $1\leqslant i\leqslant m$.

If $\la i_1,v_1 \ra,\la i_2,v_2 \ra,\ldots , \la i_t,v_t \ra, \la i_1,v_1 
\ra$ is a directed cycle in $D_2$, then $\la 0,v_1 \ra$, $\la 0,v_2 \ra$, 
$\ldots , \la 0,v_t \ra, \la 0,v_1 \ra$ is a directed closed walk in $D_1$,
contradicting the acyclicity of $D_1$. Similarly, $D_2$ has no dependent 
arc since $D_1$ has none.

{\bf Step 3}.\
Let $D_3$ be the extension of $D_2$ into ${\sf M}_m(G)$ by defining $\la 
m,w\ra \rightarrow u$ for every $\la 0,w \ra$. 

Since $D_2$ is acyclic and $u$ is a sink in $D_3$, $D_3$ is acyclic. If 
$e$ is a dependent arc in $D_3$, then $e$ must be some $\la m,w\ra 
\rightarrow u$. If $\la 0,w \ra \ne \la 0,v' \ra$, then there is a directed 
path $P'$ from $\la m,w\ra$ to a certain $\la m,w'\ra$ in $D_3$. Since there 
is no edge between $\la m,w\ra$ and $\la m,w'\ra$ in ${\sf M}_m(G)$, $P'$ 
must have length at least 2. Hence, we can find a directed path of length at 
least 2 in $D_1$ and $\la 0,v' \ra$ is not the initial vertex of that path. 
This is a contradiction.

Let us consider the arc $\la m,v'\ra \rightarrow u$. Let $\la 0,v'' \ra$ be 
a neighbor of $\la 0,v \ra$ in $Y$. The cycle $u, \la m,v'\ra, \la m-1,v\ra, 
\la m,v''\ra, u$ shows that $\la m,v'\ra \rightarrow u$ is a unique dependent 
arc in $D_3$.
\end{proof}

\bigskip

A graph $G$ satisfying Theorem \ref{mycielski d_min=1} can be constructed as 
follows. Let $v$ be a fixed vertex. Let $X$ be a set of $p \geqslant 2$ vertices
and $Y$ be a set of $q \geqslant 2$ vertices. Choose a vertex $v'$ in $X$ and 
a nonempty proper subset $Y'$ of $Y$. Add edges $vv'$ and $vv''$ for all $v'' 
\in Y'$. Add a path of length at least 3 from $v'$ to some vertex $z$ in $Y'$ 
which alternately uses vertices in $X$ and $Y$ and uses no vertex in $Y'$ except 
the terminal vertex $z$.

However, the problem of characterizing graphs $G$ that satisfy 
$d_{\min}({\sf M}_m(G))=1$ remains open.

%
\section{Generalizing a theorem of Collins and Tysdal}
%

The following appeared in Collins and Tysdal \cite{ct}.

\begin{theorem}\label{basic_mycielski}
Let $G$ be a triangle-free graph. Then the following statements hold.
\begin{enumerate}
\item   If $d_{\min}(G)\geqslant 1$, then $d_{\min}(M(G))\geqslant 3$.

\item   If $d_{\min}(G)\geqslant 2$, then $d_{\min}(M(G))\geqslant 4$.

\item   If $d_{\min}(G)\geqslant 3$, then $d_{\min}(M(G))\geqslant 6$.
\end{enumerate}
\end{theorem}

Let $S$ be a set of vertices of the graph $G(V_0, E_0)$. We use  $S'$ 
to denote the set of vertices $\{ \la 1, j \ra \mid \la 0, j \ra \in S\}$ 
and $G-S+S'$ to denote the subgraph of $M(G)$ induced by the set of 
vertices $(V_0 \setminus S) \cup S'$ in $M(G)$.

\begin{lemma}\label{replace}
If $S$ is an independent set of $G$, then the subgraph $G-S+S'$ of 
$M(G)$ is isomorphic to $G$.
\end{lemma}

\begin{proof}
The mapping $\sigma : V(G)\rightarrow V(G-S+S')$ defined below is an 
isomorphism. $\sigma(\la 0,i \ra)=\la 1,i \ra$ if $\la 0,i \ra \in S$
and $\sigma(\la 0,i \ra)=\la 0,i \ra$ if $\la 0,i \ra \notin S$.
\end{proof}

\bigskip
Proofs of Lemmas \ref{2edges} and \ref{3edges} are modeled after
ideas used in Collins and Tysdal \cite{ct}.

\begin{lemma}\label{2edges}
Let $G(V_0, E_0)$ be a triangle-free graph with at least two edges.
For any two edges $e_1, e_2$ in $M(G)-u$, $M(G)-u-\{e_1,e_2\}$ contains 
a subgraph isomorphic to $G$.
\end{lemma}

\begin{proof}
If none of $e_1$ and $e_2$ is an edge in $E_0$, we are done. Hence, 
we assume that $e_1=\la 0,x_1\ra \la 0,y_1\ra \in E_0$ and consider 
the subgraph $G'$ of $M(G)$ induced by $(V_0 \setminus \{\la 0,x_1\ra\})
\cup\{\la 1,x_1\ra\}$. The graph $G'$ is isomorphic to $G$. If $e_2$ is 
not an edge in $G'$, we are done. Assume that $e_2$ is an edge in $G'$.

{\bf Case 1.} The edge $e_2$ is not incident to $\la 1,x_1\ra$. Since 
$G$ is triangle-free, $\la 0,x_1\ra$ can not be adjacent to both endpoints 
of $e_2$. Suppose that $\la 0,x_2\ra$ is an endpoint of $e_2$ and not 
adjacent to $\la 0,x_1\ra$. Let $S=\{\la 0,x_1\ra, \la 0,x_2\ra\}$.

{\bf Case 2.} The vertex $\la 1,x_1\ra$ is an endpoint of $e_2$. Let $S=
\{\la 0,y_1\ra\}$.

In each case, $S$ is an independent set. By Lemma \ref{replace}, $G-S+S'$ 
is a subgraph of $M(G)-u-\{e_1,e_2\}$ that is isomorphic to $G$.
\end{proof}

\begin{theorem}
If a graph $G$ is triangle-free with at least two edges and $d_{\min}(G)
\geqslant 1$, then $d_{\min}({\sf M}_m(G))\geqslant d_{\min}(G)+2$.
\end{theorem}

\begin{proof}
By assumption, $d_{\min}(M(G)-u) \geqslant d_{\min}(G) \geqslant 1$. Let 
$F$ be the set of dependent arcs of an acyclic orientation $D$ of $M(G)-u$ 
that satisfies $d(D)=d_{\min}(M(G)-u)$, hence $|F| \geqslant 1$. Pick an edge
$e_1$ from $F$ and another edge $e_2 \ne e_1$ of $M(G)-u$. By Lemma 
\ref{2edges}, $M(G)-u-\{e_1,e_2\}$ contains a subgraph isomorphic to $G$. 
Thus $d_{\min}(M(G)-u) \geqslant d_{\min}(G)+1\geqslant 2$, and hence we can 
find two distinct edges $e_1'$ and $e_2'$ from $F$. By Lemma \ref{2edges} again, 
$M(G)-u-\{e_1',e_2'\}$ contains a subgraph isomorphic to $G$. It follows that 
$d_{\min}(M(G)-u) \geqslant d_{\min}(G)+2$. Finally, $d_{\min}({\sf M}_m(G))
\geqslant d_{\min}(M(G)-u) \geqslant d_{\min}(G)+2$.
\end{proof}

\bigskip

If we replace the set $F$ in the above proof by a set $F'$ of edges of 
$M(G)-u$ such that $M(G)-u-F'$ is a cover graph and $|F'|=c(M(G)-u)$, 
then we can use the same argument to get the following.
 
\begin{corollary}
If a graph $G$ is triangle-free and $c(G)\geqslant 1$, then $c({\sf M}_m(G))
\geqslant c(G)+2$.
\end{corollary}

\begin{lemma}\label{3edges}
Let $G(V_0, E_0)$ be a triangle-free graph with $\|G\|\geqslant 3$. For any 
three edges $e_1, e_2, e_3$ in $E_0$, $M(G)-u-\{e_1,e_2,e_3\}$ contains a 
subgraph isomorphic to $G$.
\end{lemma}

\begin{proof}
Let  $G'$ be the subgraph of $G$ induced by $\{e_1,e_2,e_3\}$.

{\bf Case 1.}\  If $G'$ is a star, then let $x$ be the vertex of degree 3 
and let $S=\{x\}$.

{\bf Case 2.}\  If $G'$ is a path $v_0v_1v_2v_3$ of length 3, then let
$S=\{v_0,v_2\}$. Since $G$ is triangle-free, $v_0$ and $v_2$ are not 
adjacent.

{\bf Case 3.}\  If $G'$ consists of the disjoint union of a path $P_3$
of length 2 and an edge $P_2$, then one endpoint $y$ of $P_2$ is not 
adjacent to the center vertex $x$ of $P_3$ because $G$ is triangle-free. 
Let $S=\{x,y\}$.

{\bf Case 4.}\  Let the three edges $e_1=x_1y_1$, $e_2=x_2y_2$, and $e_3
=x_3y_3$ be mutually non-incident. Since $G$ is triangle-free, at least 
one endpoint of $e_2$, say $x_2$, is not adjacent to $x_1$. Similarly,
at least one endpoint of $e_3$, say $x_3$, is not adjacent to $x_1$.

If the vertices $x_2$ and $x_3$ are not adjacent, then let $S=\{x_1,$ 
$x_2,$ $x_3\}$.

If $x_2$ and $x_3$ are adjacent and $x_1$ is not adjacent to $y_i$, 
$i = 2$ or $3$, then $y_i$ and $x_{5-i}$ are not adjacent. Let $S=\{
x_1,y_i,x_{5-i}\}$.

If $x_2$ and $x_3$ are adjacent and both $y_2$ and $y_3$ are adjacent 
to $x_1$, then $\{y_1, y_2, y_3\}$ is an independent set. Let $S=\{y_1,
y_2,y_3\}$.

In all cases, $S$ so defined is an independent set. By Lemma \ref{replace},
$G-S+S'$ is a subgraph of $M(G)-u-\{e_1,e_2,e_3\}$ isomorphic to $G$.
\end{proof}

\begin{theorem}
If a graph $G$ is triangle-free with at least three edges and $d_{\min}(G)
\geqslant 3$, then $d_{\min}({\sf M}_m(G))\geqslant d_{\min}(G)+3$.
\end{theorem}

\begin{proof}
Let $F$ be the set of dependent arcs of an acyclic orientation $D$ of $G$ 
that satisfies $d(D)=d_{\min}(G)$, hence $|F| \geqslant 3$. Pick three 
edges $e_1, e_2, e_3$ from $F$. By Lemma \ref{3edges}, $M(G)-u-\{e_1,e_2,
e_3\}$ contains a subgraph isomorphic to $G$. It follows that 
$d_{\min}({\sf M}_m(G))\geqslant d_{\min}(M(G)-u) \geqslant d_{\min}(G)+3$.
\end{proof}

\begin{corollary}
If a graph $G$ is triangle-free with at least three edges and $c(G)\geqslant 
3$, then $c({\sf M}_m(G))\geqslant c(G)+3$.
\end{corollary}

%
\section{Upper bounds of $c({\sf M}_m(G))$}
%

In this section, we derive upper bounds for $c({\sf M}_m(G))$. Since 
$\chi(G)<g(G)$ implies that $G$ is a cover graph, we have the following 
inequality.
$$c(G)\leqslant \min \{\|G\|-\|H\| \mid \mbox{$H$ is a subgraph of $G$ 
and } \chi(H)<g(H)\}.$$

Let $e_k(G)$ be the maximum number of edges in a $k$-colorable subgraph 
of $G$. Since the girth of a subgraph is never smaller than that of the 
given graph, the above inequality implies the following.
\begin{equation}\label{girth and e_k}
c(G) \leqslant \|G\| - e_{k-1}(G) \mbox{ if } g(G) \geqslant k.
\end{equation}

Let $G$ be a triangle-free graph. If $H=(X,Y)$ is a bipartite subgraph of 
$G$, then the following inequality holds by the above inequality.
\begin{equation}\label{girth>=4 pi_E}
c(G) \leqslant \|G\|- \|H\| -e_2(G[X]).
\end{equation}

\medskip

\begin{proof}
Let $X'=(X_1,X_2)$ be a bipartite subgraph of $G[X]$ with $e_2(G[X])$ 
edges. Consider the subgraph $G'=(V(H),E(H)\cup E(X'))$ of $G$.
Obviously, $G'$ is 3-colorable. Hence, $e_3(G)\geq \|H\|+\|X'\|=\|H\|+
e_2(G[X])$. By inequality (\ref{girth and e_k}), we are done.
\end{proof}

\begin{theorem}
If $G$ is a graph and $m$ is a positive integer, then $c({\sf M}_m(G)) 
\leqslant \|G\|$. Moreover, if $G$ is a triangle-free graph, then 
$c({\sf M}_m(G)) \leqslant \|G\|-e_2(G)$.
\end{theorem}

\begin{proof}
Obviously, ${\sf M}_m(G)-E(G)$ is bipartite. We have $e_2({\sf M}_m(G))$ 
$\geqslant$  $\|{\sf M}_m(G)\|-\|G\|$. By inequality (\ref{girth and e_k}), 
$c({\sf M}_m(G))\leq \|G\|$. If $G$ is triangle-free, so is ${\sf M}_m(G)$.
It is easy to see that ${\sf M}_m(G)-E(G)$ is a bipartite graph with 
bipartition $(X,Y)$ such that the vertices of $G$ belong to the same 
partite set, say $X$. Since $G$ has a  bipartite subgraph with $e_2(G)$ 
edges, $e_2({\sf M}_m(G)[X])\geqslant e_2(G)$. By inequality (\ref{girth>=4 
pi_E}), $c({\sf M}_m(G)) \leqslant \|{\sf M}_m(G)\|- \|{\sf M}_m(G)-E(G)\| 
-e_2({\sf M}_m(G)[X])\leqslant \|G\|-e_2(G)$.
\end{proof}

\bigskip

{\bf Acknowledgment}.\
The authors are indebted to Dr. Fei-Huang Chang for useful discussions
on Theorem \ref{mycielski d_min=1}.



\begin{thebibliography}{99}

\bibitem{bbn} B. Bollob\'{a}s, G. Brightwell, J. Ne\v{s}et\v{r}il,
Random graphs and covering graphs of posets,
Order 3 (1986), 245-255.

\bibitem{ct} K. L. Collins, K. Tysdal, 
Dependent edges in Mycielski graphs and 4-colorings of 4-skeletons, 
J. Graph Theory  46 (2004),  285-296.

\bibitem{fflw}D. C. Fisher, K. Fraughnaugh, L. Langley, D. B. West,
The number of dependent arcs in an acyclic orientation,
J. Combin. Theory Ser. B 71 (1997), 73-78.

\bibitem{ll} H.-H. Lai, K.-W. Lih,
On preserving full orientability of graphs,
European J. Combin.  31 (2010), 598-607.

\bibitem{tower} K.-W. Lih, C.-Y. Lin, L.-D. Tong, 
Non-cover generalized Mycielski, Kneser, and Schrijver graphs, 
Discrete Math. 308 (2008), 4653-4659.

\bibitem{mo} K. M. Mosesian,
Some theorems on strongly basable graphs,
Akad. Nauk. Armian. SSR. Dokl. 54 (1972), 241-245.
(in Russian)

\bibitem{pret} O. Pretzel, 
On graphs that can be oriented as diagrams of ordered sets, 
Order 2 (1985), 25-40.

\bibitem{pret86} O. Pretzel, 
On reorienting graphs by pushing down maximal vertices, 
Order 3 (1986), 135-153.

\bibitem{pret03} O. Pretzel,  
On reorienting graphs by pushing down maximal vertices II,
Discrete Math. 270 (2003), 227--240.

\bibitem{py} O. Pretzel, D. Youngs, 
Cycle lengths and graph orientations, 
SIAM J. Discrete Math. 3 (1990), 544-553. 

\bibitem{rt} V. R\"{o}dl, L. Thoma, 
On cover graphs and dependent arcs in acyclic orientations, 
Combin. Probab. Comput. 14 (2005), 585-617.

\end{thebibliography}
\end{document}